\documentclass[
final
 , nomarks
]{dmtcs-episciences}

\usepackage[utf8]{inputenc}
\usepackage{subfigure}
\usepackage{amsmath,amsthm}

\usepackage{algorithm}
\usepackage{algorithmicx}
\usepackage{algpseudocode}

\theoremstyle{plain}
\newtheorem{thm}{Theorem}[section]
\newtheorem{cor}{Corollary}[section]
\newtheorem{lem}{Lemma}[section]
\theoremstyle{definition}
\newtheorem{defn}[thm]{Definition}
\theoremstyle{plain}
\newtheorem{conj}{Conjecture}[section]
\usepackage[round]{natbib}

\author{Xiao-Lu Gao%\thanks{E-mail address: gaoxl14@lzu.edu.cn.}
  \and Shou-Jun Xu\thanks{Corresponding author.}}
\title[The LexCycle on $\overline{P_{2}\cup P_{3}}$-free Cocomparability  Graphs]{The LexCycle on $\overline{P_{2}\cup P_{3}}$-free Cocomparability  Graphs\footnote{This work is supported by NSFC (Grant Nos. 12071194, 11571155).}}
\affiliation{School of Mathematics and Statistics, Lanzhou University, Gansu, China}
\received{2019-04-18}
\revised{2020-03-31, 2020-12-03}
\accepted{2020-12-09}
\begin{document}
\publicationdetails{22}{2020}{4}{13}{5358}
\maketitle
\begin{abstract}
  A graph $G$ is a cocomparability graph if there exists an acyclic transitive orientation of the edges of its complement graph $\overline{G}$. LBFS$^{+}$ is a variant of the generic Lexicographic Breadth First Search (LBFS), which uses a specific tie-breaking mechanism. Starting with some ordering $\sigma_{0}$ of $G$, let $\{\sigma_{i}\}_{i\geq 1}$ be the sequence of orderings such that $\sigma_{i}=$LBFS$^{+}(G, \sigma_{i-1})$. The LexCycle($G$) is defined as the maximum length of a cycle of vertex orderings of $G$ obtained via such a sequence of LBFS$^{+}$ sweeps. Dusart and Habib conjectured in 2017 that LexCycle($G$)=2 if $G$ is a cocomparability graph and proved it holds for interval graphs. In this paper, we show that LexCycle($G$)=2 if $G$ is a $\overline{P_{2}\cup P_{3}}$-free cocomparability graph, where a $\overline{P_{2}\cup P_{3}}$ is the graph whose complement is the disjoint union of $P_{2}$ and $P_{3}$. As corollaries, it's applicable for diamond-free cocomparability graphs, cocomparability graphs with girth at least 4, as well as interval graphs.
\end{abstract}

\section{Introduction}
\label{sec:in}

Lexicographic Breadth First Search (LBFS) is a graph search paradigm which was developed by Rose, Tarjan and Lueker \cite{1st-LBFS76} for providing a simple linear time algorithm to recognize chordal graphs, namely, graphs containing no induced cycle of length greater than three. Since then, researchers have done plenty of studies on the properties and applications of LBFS \cite{LBFSeg1, LBFSeg2}. At each step of an LBFS procedure, a vertex is visited only if it has the lexicographically largest label. If there exists more than one such eligible vertex at some step, these vertices are said to be {\em tied} at this step.

A multi-sweep algorithm is an algorithm that produces a sequence of orderings $\{\sigma_{i}\}_{i\geq 0}$ where each ordering $\sigma_{i} (i \geq 1)$ breaks ties using specified tie-breaking rules by referring to the previous ordering $\sigma_{i-1}$. In particular, LBFS$^{+}$ is one of the most widely used variants of LBFS, which is a multi-sweep algorithm that chooses the rightmost tied vertex in the previous sweep, and therefore produces a unique vertex ordering. It was first investigated in \cite{LBFS+investigate1, LBFS+investigate2}, and has been used to recognize several well-known classes of graphs, such as unit interval graphs \cite{LBFSrecognize1}, interval graphs \cite{LBFSrecognize2, LBFS-recog-Lipeng} and cocomparability graphs \cite{2015}. Here we present a description of the generic LBFS procedure in Algorithm \ref{LBFS algorithm} which starts with a distinguished vertex and then allows arbitrary tie-breaking; following the LBFS procedure we impose the specific tie-breaking mechanism LBFS$^{+}$ in Algorithm \ref{LBFS+algorithm}.

\begin{algorithm}[h]
\caption{LBFS ($G, v$)}\label{LBFS algorithm}
\begin{algorithmic}[1]
    \Require
      a graph $G(V, E)$ and a distinguished vertex $v$ of $G$
    \Ensure
      an ordering $\sigma_{v}$ of vertices of $G$
\State label($v$) $\leftarrow$ $|V|$
\State assign the label $\epsilon$ to all the vertices of $V-\{v\}$
\For {$i\leftarrow$ 1 to $|V|$}
\State pick any unnumbered vertex $u$ with the lexicographically largest label \hfill{($\S$)}
\State $\sigma_{i} \gets u$
\For {each unnumbered vertex $w\in N(u)$}
\State append $(n-i)$ to label($w$)
\EndFor
\EndFor\\
\Return $\sigma_{v}$
\end{algorithmic}
\end{algorithm}

\smallskip
\begin{algorithm}
\caption{LBFS$^{+}(G, \pi)$}\label{LBFS+algorithm}
\begin{algorithmic}[1]
    \Require
      a graph $G(V, E)$ and an ordering $\pi$ of vertices of $G$
    \Ensure
      an ordering $\sigma$ of vertices of $G$

\noindent We run LBFS($G, \pi(|V|)$). In step ($\S$) of the LBFS procedure, let $L$ be the set of unnumbered vertices with the lexicographically largest label. Choose $u$ to be the vertex in $L$ that appears rightmost in $\pi$.
%\Return $\sigma$
\end{algorithmic}
\end{algorithm}

The LexCycle of a graph $G$ is the size of the longest cycle resulting from a series of LBFS$^{+}$'s on $G$. Since a finite graph has a finite number of vertex orderings, this series will converge to a number of fixed orderings that produce a cycle, the largest size of which is captured by this LexCycle parameter. Charbit et al. first introduced this new graph parameter \cite{LexCycle}. They believed that a small LexCycle often leads to a linear structure that has been exploited algorithmically on a number of graph classes.

\smallskip
\begin{defn}\cite{LexCycle}
  Let $G$ be a graph, the {\em LexCycle$(G)$} is defined as the maximum length of a cycle of vertex orderings of $G$ obtained via a sequence of LBFS$^{+}$ sweeps starting with an arbitrary vertex ordering of $G$.
\end{defn}

Comparability graphs are the graphs that admit an acyclic transitive orientation of the edges; that is, there is an orientation of the edges such that for every three vertices $x, y, z$, if the edges $xy, yz$ are oriented $x\rightarrow y\rightarrow z$, then $xz\in E$ and $x\rightarrow z$. Cocomparability graphs are the complement graphs of comparability graphs and have been widely studied \cite{cocomp-domination, cocomp-independent, cocomp-2016, Mertzios12, LDFSeg2, LDFS-MPC13}. The well-studied interval graphs, co-bipartite graphs, permutation graphs and trapezoid graphs are subclasses of cocomparability graphs; and both comparability graphs and cocomparability graphs are well-known subclasses of perfect graphs \cite{Golumbic04}.

Charbit et al. \cite{LexCycle} reintroduced the conjecture that LexCycle($G$)=2 if $G$ is a cocomparability graph which was firstly raised in \cite{2015}. In particular, they showed that LexCycle($G$)=2 for some subclasses of cocomparability graphs (proper interval, interval, co-bipartite, domino-free cocomparability graphs) as well as trees. They mentioned that to prove the conjecture, a good way is to start by proving that it holds for $k$-ladder-free cocomparability graphs for any positive integer $k$. Further, they conjectured that LexCycle($G$)=2 even for AT-free graphs, which strictly contain cocomparability graphs. The $k$-ladder and asteroidal triple (AT) will be introduced in section 2.

In this paper, we show that LexCycle($G$)=2 for $\overline{P_{2}\cup P_{3}}$-free cocomparability graphs, i.e., Theorem \ref{LexCycle-result}. These graphs strictly contain interval graphs and are unrelated under inclusion to domino-free cocomparability graphs, where the LexCycle of any of these two graphs is proved to be 2 in \cite{LexCycle}. The rest of this paper is organized as follows. We present in section 2 some preliminary definitions, notations and known results. In section 3, we present the main results. In the final section we present concluding remarks.

\section{Preliminaries}
\label{sec:first}

In this paper, we consider simple finite undirected graphs $G=(V, E)$ on $n=|V|$ vertices. An {\em ordering} $\sigma$ of $G$ is a bijection $\sigma$ from $V$ to $\{1, 2, ..., n\}$. We write $u\prec_{\sigma} v$ if and only if $\sigma(u)< \sigma(v)$ and $u$ is said to be to the {\em left} of $v$ in $\sigma$ if $u\prec_{\sigma} v$. Given a sequence of orderings $\{\sigma_{i}\}_{i\geq 0}$, we write $u\prec_{i} v$ if $u\prec_{\sigma_{i}} v$; and $u\prec_{i, j} v$ if both $u\prec_{\sigma_{i}} v$ and $u\prec_{\sigma_{j}} v$. For $S\subseteq V$, the {\em induced subgraph} $G[S]$ of $G$ is the graph whose vertex set is $S$ and whose edge set consists of all the edges in $E$ with both end-vertices in $S$; we write $\sigma[S]$ to denote the ordering of $\sigma$ restricted to the vertices of $S$. $G$ is called {\em $H$-free} if $G$ does not contain $H$ as an induced subgraph. $P_{n}$ and $C_{n}$ denote a path and cycle respectively on $n$ vertices. A {\em domino} is a pair of $C_{4}$'s sharing an edge. The {\em girth} $g(G)$ of $G$ is the minimum length of a cycle in $G$ ($g(G)=\infty$ if $G$ does not contain a cycle). A {\em $k$-ladder} is a graph $G$ with $V(G)=\{a, a_{1}, a_{2}, ..., a_{k}$, $b, b_{1}, b_{2}, ..., b_{k}\}$ and edge set $E(G)= \{ab, aa_{1}, bb_{1}\} \cup \{a_{j}b_{j}| 1\leq j\leq k\} \cup \{a_{j}a_{j+1}, b_{j}b_{j+1}| 1\leq j\leq k-1\}$, as shown in Fig. \ref{Fig:k-ladder}.
\smallskip
\begin{figure}[ht]
          \centering
          \includegraphics[width=2.6in]{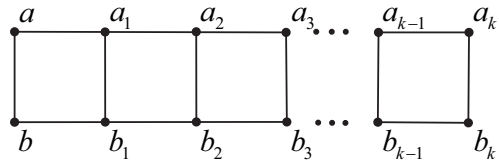}
          \caption{\small A $k$-ladder.}\label{Fig:k-ladder}
  \end{figure}

\smallskip
\begin{thm}{\rm\cite{1stLDFS08}} \label{cocomp ordering}
A graph $G=(V, E)$ is a cocomparability graph if and only if there exists a vertex ordering $\sigma$ such that if $x\prec_{\sigma} y\prec_{\sigma} z$ and $xz\in E$, then either $xy\in E$ or $yz\in E$ or both.
\end{thm}

Such an ordering in Theorem \ref{cocomp ordering} is called a {\em cocomparability ordering}, or an {\em umbrella-free ordering}. $G$ is an {\em interval graph} if its vertices can be put in one-to-one correspondence with intervals on the real line such that two vertices are adjacent in $G$ if and only if the corresponding intervals intersect.  An {\em asteroidal triple} (AT) is an independent triple of vertices $u, v, w$ such that every pair of the triple is connected when removing the closed neighbourhood of the third vertex from the graph.

\smallskip
There is a nice vertex ordering characterization of LBFS as shown in Lemma \ref{LBFS}, known as the {\em 4-Point Condition}, which plays a key role in the proof of the correctness of our result.

\begin{lem}{\rm\cite{1stLDFS08}} \label{LBFS}
  {\bf (4-Point Condition)} A vertex ordering $\sigma$ of a graph $G$ with vertex set $V$ is an LBFS ordering if and only if for any triple $x\prec_{\sigma} y\prec_{\sigma} z$, where $xz\in E$ and $xy\notin E$, there exists a vertex $w\prec_{\sigma} x$ such that $wy\in E$ and $wz\notin E$.
\end{lem}

\begin{figure}[ht]
          \centering
          \includegraphics[width=3.5in]{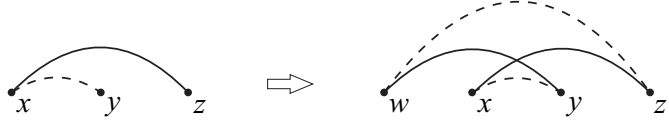}
          \caption{\small The 4-Point Condition.}\label{Fig:4-point-condition}
\end{figure}

Given a pair of vertices $y$ and $z$, we call a vertex $w$ where $wy\in E$ and $wz\notin E$ a {\em private neighbour} of $y$ with respect to $z$. A triple $(x, y, z)$ satisfying $x\prec_{\sigma} y\prec_{\sigma} z$ where $xz\in E$ and $xy\notin E$ is called a {\em bad triple} with respect to $\sigma$, where $\sigma$ is an ordering of $G$. In this paper, we always choose the vertex $w$ in Lemma \ref{LBFS} as the {\em leftmost} private neighbour of $y$ with respect to $z$ in $\sigma$, and write it as $w=$LMPN$(y|_{\sigma}z)$.

\smallskip
It follows directly from Theorem \ref{cocomp ordering} and Lemma \ref{LBFS} that an LBFS cocomparability ordering satisfies the following property.

\begin{thm}{\rm\cite{LexCycle}}\label{LBFS C4 Property}
  {\bf (LBFS $C_{4}$ Property)} Let $G=(V, E)$ be a cocomparability graph and $\sigma$ an LBFS cocomparability ordering of $G$. Then for every triple $x\prec_{\sigma} y\prec_{\sigma} z$ with $xz\in E$ and $xy\notin E$, there exists a vertex $w\prec_{\sigma} x$ such that $\{w, x, y, z\}$ induces a cycle where $wx, wy, yz\in E$.
\end{thm}

\begin{figure}[ht]
          \centering
          \includegraphics[width=3.5in]{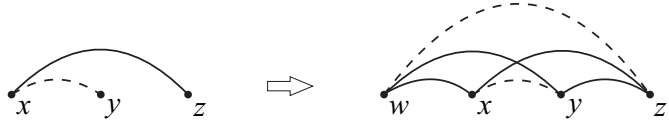}
          \caption{\small The LBFS $C_{4}$ Property.}\label{Fig:LBFS-C4}
\end{figure}

\begin{thm}{\rm\cite{LDFSeg2}} \label{LBFS cocomp}
  Let $G$ be a cocomparability graph, and $\pi$ a cocomparability ordering of $G$. Then the LBFS ordering $\sigma$=LBFS$^{+}(\pi)$ is also a cocomparability ordering of $G$.
\end{thm}

Dusart and Habib presented a simple multi-sweep algorithm called {\em Repeated} LBFS$^{+}$, where the algorithm Repeated LBFS$^{+}$ starts with an arbitrary LBFS ordering $\sigma_{1}$ and produces $n=|V(G)|$ consecutive LBFS orderings $\sigma_{i}$($1\leq i\leq n$) such that $\sigma_{i}$=LBFS$^{+}(\sigma_{i-1})$ for $2\leq i\leq n$. The authors proved in \cite{2015} that $G$ is a cocomparability graph if and only if the Repeated LBFS$^{+}$ algorithm computes a cocomparability ordering. They further conjectured that this series always falls into a cycle of length 2. We state these results below.

\smallskip
\begin{lem}{\rm \cite{2015}}\label{Repeated LBFS+}
  $G$ is a cocomparability graph if and only if $\mathcal{O}(n)$ LBFS$^{+}$ sweeps compute a cocomparability ordering.
\end{lem}

\begin{conj}{\rm \cite{2015}}\label{conj}
  If $G$ is a cocomparability graph, then $LexCycle(G)=2$.
\end{conj}

The conjecture is formulated based on the easy but very important tool called the {\em Flipping Lemma} about LBFS on cocomparability graphs.

\smallskip
\begin{lem}{\rm \cite{LDFSeg2}}
  {\bf (The Flipping Lemma)} Let $G=(V, E)$ be a cocomparability graph, $\sigma$ a cocomparability ordering of $G$ and $\tau=LBFS^{+}(\sigma)$. Then for every non-edge $uv\notin E$, $u\prec_{\sigma} v\Leftrightarrow v\prec_{\tau} u$.
\end{lem}

\section{Main results}
\label{sec:hints}
This paper presents a proof of a subcase of Conjecture \ref{conj}. In the following we will show that LexCycle(G)=2 where $G$ is a $\overline{P_{2}\cup P_{3}}$-free cocomparability graph. The graph $\overline{P_{2}\cup P_{3}}$ is shown in Fig. \ref{Fig:P2-P3}.

\smallskip
\begin{figure}[ht]
          \centering
          \includegraphics[width=1.2in]{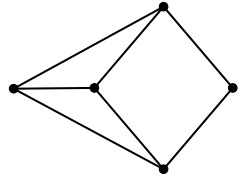}
          \caption{\small the graph $\overline{P_{2}\cup P_{3}}$.}\label{Fig:P2-P3}
  \end{figure}

\smallskip
\begin{thm} \label{result}
  Let $G$ be a $\overline{P_{2}\cup P_{3}}$-free cocomparability graph, $\pi$ an arbitrary cocomparability ordering of $G$, and
   $\{\sigma_{i}\}_{i\geq 0}$ a sequence of LBFS$^{+}$ orderings where $\sigma_{i+1}=LBFS^{+}(\sigma_{i})$ and $\sigma_{0}=LBFS^{+}(\pi)$. Then $\sigma_{1}=\sigma_{3}$.
\end{thm}
\begin{proof}
  We prove this theorem by contradiction. We will show an infinite structure of $G$, which is a contradiction to the finiteness of $G$.

  Since $\pi$ is a cocomparability ordering of $G$, it follows from Theorem \ref{LBFS cocomp} that each ordering $\sigma_{i}$ ($i\geq 0$) is an LBFS cocomparability ordering of $G$. Suppose to the contrary that $\sigma_{1}\neq \sigma_{3}$. Let $\sigma_{1}=u_{1}, u_{2}, ..., u_{n}$ and $\sigma_{3}=v_{1}, v_{2}, ..., v_{n}$. Denote $k$ the index of the leftmost vertex where $\sigma_{1}$ and $\sigma_{3}$ differ. Let $a_{1}$ (resp. $b_{1}$) denote the $k^{\rm{th}}$ vertex of $\sigma_{1}$ (resp. $\sigma_{3}$). Then $u_{i}=v_{i}$ for any $i< k$ and $u_{k}=a_{1}$, $v_{k}=b_{1}$. Thus $a_{1}\prec_{1} b_{1}$ and $b_{1}\prec_{3} a_{1}$. The following claim presents the infinite structure of $G$.

  \smallskip
%  \begin{claim}\label{induction}
{\it \textbf{Claim 1.}}
    {\em Assume that $a_{1}, b_{1}$ were given as defined previously. Then, for any integer $t\geq 2$, there always exists a $(t-1)$-ladder with vertex set $\{a_{1}, a_{2}, ..., a_{t}$, $b_{1}, b_{2}, ..., b_{t}\}$, satisfying that$:$

    $(1)$ $a_{j+1}$=LMPN$(a_{j}|_{\sigma_{2}} b_{j})$, $b_{j+1}=$ LMPN$(b_{j}|_{\sigma_{0}} a_{j}), \forall \,1\leq j \leq t-1$.

    $(2)$ $E(G[\{a_{1}, a_{2}, ..., a_{t}, b_{1}, b_{2}, ..., b_{t}\}])= \{a_{j}b_{j}| 1\leq j\leq t\} \cup \{a_{j}a_{j+1}, b_{j}b_{j+1}| 1\leq j\leq t-1\}$.

    $(3)$ $b_{t}\prec_{0} a_{t}\prec_{0} b_{t-1}\prec_{0} a_{t-1}\prec_{0} ... \prec_{0} b_{1}\prec_{0} a_{1};\,
         a_{1}\prec_{1} b_{1}\prec_{1} a_{2}\prec_{1} b_{2}\prec_{1} ... \prec_{1} a_{t}\prec_{1} b_{t};$\\
         $a_{t}\prec_{2} b_{t}\prec_{2} a_{t-1}\prec_{2} b_{t-1}\prec_{2} ... \prec_{2} a_{1}\prec_{2} b_{1}.$}%;\,
         %b_{1}\prec_{3} a_{1}\prec_{3} b_{2}\prec_{3} a_{2}\prec_{3} ... \prec_{3} b_{i}\prec_{3} a_{i}.$}

\smallskip
  We prove the claim by induction on $t$. We first show it holds for the base case $t=2$.

  Let $S=\{u_{1}, u_{2}, ..., u_{k-1}\}=\{v_{1}, v_{2}, ..., v_{k-1}\}$ ($S$ might be empty), then $\sigma_{1}[S]=\sigma_{3}[S]$. Since at the time $a_{1}$ was chosen in $\sigma_{1}$ after the ordering of $S$, $b_{1}$ was simultaneously chosen in $\sigma_{3}$, it follows that label($a$)=label($b$) at iteration $k$ in both $\sigma_{1}$ and $\sigma_{3}$, i.e., $S\cap N(a_{1})=S\cap N(b_{1})$. Therefore when $a_{1}$ was chosen in $\sigma_{1}$, the ``${+}$'' rule was applied to break ties between $a_{1}$ and $b_{1}$ and so $b_{1}\prec_{0} a_{1}$. Similarly, we have $a_{1}\prec_{2} b_{1}$.

  Since $a_{1}\prec_{1, 2} b_{1}$, we know that $a_{1}b_{1}\in E$ (by the Flipping Lemma) and there exists a vertex left of $a_{1}$ in $\sigma_{2}$ which is a private neighbour of $a_{1}$ with respect to $b_{1}$. We choose $a_{2}$ as $a_{2}$=LMPN$(a_{1}|_{\sigma_{2}} b_{1})$. Using the Flipping Lemma on the non-edge $a_{2}b_{1}$, we place $a_{2}$ in the remaining orderings and obtain that $a_{2}\prec_{0} b_{1}$, $b_{1}\prec_{1} a_{2}$. This gives rise to a bad triple in $\sigma_{0}$ where $a_{2}\prec_{0} b_{1}\prec_{0} a_{1}$ and $a_{2}a_{1}\in E, a_{2}b_{1}\notin E$.
  % as in Figure \ref{Fig:initial-4-vertices}.

  By the LBFS $C_{4}$ Property, we choose vertex $b_{2}$ as $b_{2}$=LMPN$(b_{1}|_{\sigma_{0}} a_{1})$ and thus $b_{2}a_{2}\in E$. We again use the Flipping Lemma on $b_{2}a_{1}\notin E$ to place $b_{2}$ in the remaining orderings, and obtain that $a_{1}\prec_{1} b_{2}$, $b_{2}\prec_{2} a_{1}$.

  Consider the position of $b_{2}$ in $\sigma_{2}$. If $b_{2}\prec_{2} a_{2}$, then ($b_{2}, a_{1}, b_{1}$) is a bad triple, contradicting to the choice of $a_{2}$ as $a_{2}$=LMPN$(a_{1}|_{\sigma_{2}} b_{1})$. Therefore $a_{2}\prec_{2} b_{2}\prec_{2} a_{1}$, as shown in Fig. \ref{Fig:initial-4-vertices}.

  \begin{figure}[ht]
          \centering
          \includegraphics[width=5.5in]{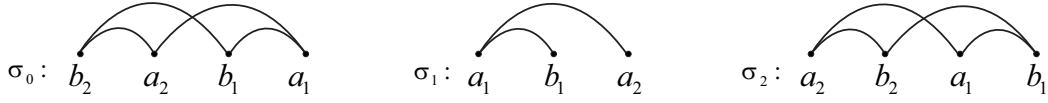}
          \caption{\small initial positions of $a_{1}, b_{1}$ in $\sigma_{0}$, $\sigma_{1}$ and $\sigma_{2}$, respectively.}\label{Fig:initial-4-vertices}
  \end{figure}

  Now we consider the position of $b_{2}$ in $\sigma_{1}$. We know that $a_{1}\prec_{1}b_{2}$. This gives rise to three cases: (\romannumeral1) $a_{1}\prec_{1}b_{2}\prec_{1}b_{1}$, or (\romannumeral2) $b_{1}\prec_{1}b_{2}\prec_{1}a_{2}$, or (\romannumeral3) $a_{2}\prec_{1}b_{2}$.

  (\romannumeral1) If $a_{1}\prec_{1}b_{2}\prec_{1}b_{1}$, then ($a_{1}, b_{2}, b_{1}$) is a bad triple in $\sigma_{1}$. Thus there exists a vertex $c\prec_{1} a_{1}$ (thus $c\in S$) such that $cb_{2}, ca_{1}\in E$ and $cb_{1}\notin E$, contradicting that $S\cap N(a_{1})=S\cap N(b_{1})$.

  (\romannumeral2) If $b_{1}\prec_{1}b_{2}\prec_{1}a_{2}$, then ($a_{1}, b_{2}, a_{2}$) is a bad triple in $\sigma_{1}$. Thus there exists a vertex $c\prec_{1} a_{1}$ (thus $c\in S$) such that $cb_{2}, ca_{1}\in E$ and $ca_{2}\notin E$. Since $S\cap N(a_{1})=S\cap N(b_{1})$, $cb_{1}\in E$. Then $\{c, a_{1}, b_{1}, b_{2}, a_{2}\}$ induces a $\overline{P_{2}\cup P_{3}}$, where $P_{2}$ is the path $a_{1}-b_{2}$ and $P_{3}$ is the path $c-a_{2}-b_{1}$, a contradiction.

  Therefore, $b_{2}$ must be placed as $a_{2}\prec_{1}b_{2}$, and thus we have completely determined the positions of vertices of $\{a_{1}, a_{2}, b_{1}, b_{2}\}$ in $\sigma_{0}, \sigma_{1}$ and $\sigma_{2}$, respectively, as shown in Fig. \ref{Fig:4-vertices-complete}. Therefore it holds for the base case.% Actually we have a more general vertex ordering property, which is raised in the following claim.

  \begin{figure}[ht]
          \centering
          \includegraphics[width=5.5in]{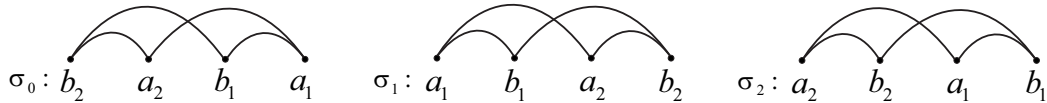}
          \caption{\small positions of $a_{2}, b_{2}$ in $\sigma_{0}$, $\sigma_{1}$ and $\sigma_{2}$, respectively.}\label{Fig:4-vertices-complete}
  \end{figure}
  From now on we suppose that it is true for $t=i$ and will prove the case when $t=i+1$. By the inductive hypothesis, there exists a sequence of vertices $a_{1}, a_{2}, ..., a_{i}, b_{1}, b_{2}, ..., b_{i}$ satisfying the three conditions in Claim 1.

          Since $a_{i}\prec_{1, 2} b_{i}$, there exists a vertex left of $a_{i}$ in $\sigma_{2}$ which is adjacent to $a_{i}$ but not to $b_{i}$. Choose $a_{i+1}$ as $a_{i+1}$=LMPN$(a_{i}|_{\sigma_{2}} b_{i})$. Using the Flipping Lemma on the non-edge $b_{i}a_{i+1}$, we have that $a_{i+1}\prec_{0} b_{i}$, $b_{i}\prec_{1} a_{i+1}$. This gives rise to a bad triple $(a_{i+1}, b_{i}, a_{i})$ in $\sigma_{0}$ where $a_{i+1}a_{i}\in E$ and $a_{i+1}b_{i}\notin E$.

          Choose $b_{i+1}$ as $b_{i+1}$=LMPN$(b_{i}|_{\sigma_{0}} a_{i})$. Using the Flipping Lemma on the non-edge $b_{i+1}a_{i}$, we obtain that $b_{i+1}\prec_{2} a_{i}$, $a_{i}\prec_{1} b_{i+1}$. If $b_{i+1}\prec_{2} a_{i+1}$, then ($b_{i+1}, a_{i}, b_{i}$) is a bad triple, contradicting that $a_{i+1}$=LMPN$(a_{i}|_{\sigma_{2}} b_{i})$. Thus, we have that $a_{i+1}\prec_{2} b_{i+1}\prec_{2} a_{i}$, as shown in Fig. \ref{Fig:initial a(i+1),b(i+1)}.

          \begin{figure}[ht]
          \centering
          \includegraphics[width=5.5in]{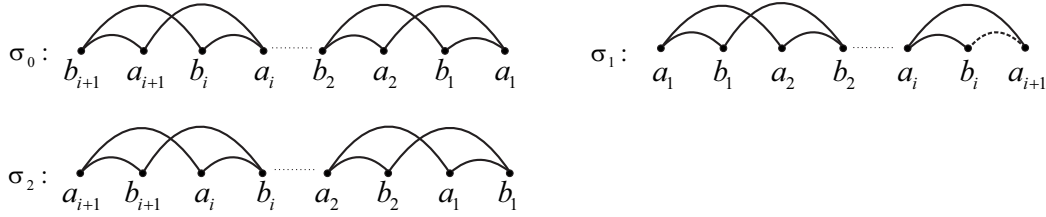}
          \caption{\small initial positions of $a_{i+1}, b_{i+1}$ in $\sigma_{0}$, $\sigma_{1}$ and $\sigma_{2}$, respectively.}\label{Fig:initial a(i+1),b(i+1)}
          \end{figure}

          Now we show that $b_{i+1}$ is adjacent to none of the vertices $a_{1}, a_{2}, ..., a_{i-1},$ $b_{1}, b_{2}, ..., b_{i-1}$, and similarly, we show that $a_{i+1}$ is adjacent to none of these vertices. We have shown that $b_{i+1}a_{i}\notin E$. Since also that $a_{i}b_{j}\notin E$ and $b_{i+1}\prec_{0} a_{i}\prec_{0} b_{j}$ for any $1\leq j\leq i-1$, it follows that $b_{i+1}b_{j}\notin E$ (by the definition a cocomparability ordering) for any $1\leq j\leq i-1$. Similarly, $a_{i+1}a_{j}\notin E$ for any $1\leq j\leq i-1$. It holds that $b_{i+1}a_{i-1}\notin E$, since otherwise, ($b_{i+1}, b_{i-1}, a_{i-1}$) is a bad triple in $\sigma_{0}$, contradicting that $b_{i}$=LMPN$(b_{i-1}|_{\sigma_{0}} a_{i-1})$. On the other hand, since $b_{i+1}\prec_{0} a_{i}\prec_{0} a_{j}$ and $b_{i+1}a_{i}, a_{i}a_{j}\notin E$ for any $1\leq j\leq i-2$, we have that $b_{i+1}a_{j}\notin E$ for any $1\leq j\leq i-2$. Therefore $b_{i+1}a_{j}\notin E$, for any $1\leq j\leq i-1$. Similarly, we deal with $a_{i+1}$ in $\sigma_{2}$ and obtain that $a_{i+1}b_{j}\notin E$, for any $1\leq j\leq i-1$. So far, we have proved the correctness of conditions (1) and (2).

          \smallskip
          What remains to be shown is the position of $b_{i+1}$ in the ordering $\sigma_{1}$. We know that $a_{i}\prec_{1}b_{i+1}$. This gives rise to three cases: (\romannumeral1) $a_{i}\prec_{1}b_{i+1}\prec_{1}b_{i}$, or (\romannumeral2) $b_{i}\prec_{1}b_{i+1}\prec_{1}a_{i+1}$, or (\romannumeral3) $a_{i+1}\prec_{1}b_{i+1}$. We will show that $b_{i+1}$ must be placed as in (\romannumeral3).

          %情形1
          \smallskip
          {\it \bf{Case 1:}} $a_{i}\prec_{1}b_{i+1}\prec_{1}b_{i}$.

          \begin{figure}[ht]
          \centering
          \includegraphics[width=6in]{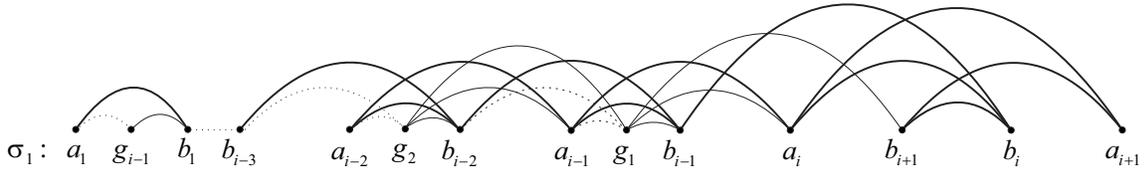}
          \caption{\small Case 1. $a_{i}\prec_{1}b_{i+1}\prec_{1}b_{i}$.}\label{Fig:b(i+1)1}
          \end{figure}

          In this case, ($b_{i-1}, b_{i+1}, b_{i}$) is a bad triple in $\sigma_{1}$, so we choose $g_{1}$ as $g_{1}$=LMPN$(b_{i+1}|_{\sigma_{1}} b_{i})$ and so $g_{1}b_{i-1}\in E$. The ladder structure implies that $g_{1}$ can't be any of the vertices $\{a_{j}, b_{j}\}_{1\leq j\leq i}$. Because of the fact that $g_{1}\prec_{1} a_{i}\prec_{1} b_{i+1}$ and $a_{i}b_{i+1}\notin E$, we have $g_{1}a_{i}\in E$. If $g_{1}\prec_{1} a_{i-1}$, then since $a_{i-1}b_{i+1}\notin E$, $g_{1}a_{i-1}\in E$. Thus, $\{g_{1}, a_{i}, a_{i-1}, b_{i-1}, b_{i}\}$ induces a $\overline{P_{2}\cup P_{3}}$, where $P_{2}$ is the path $a_{i}-b_{i-1}$ and $P_{3}$ is the path $g_{1}-b_{i}-a_{i-1}$, a contradiction. Therefore, $a_{i-1}\prec_{1} g_{1}\prec_{1} b_{i-1}$ and $g_{1}a_{i-1}\notin E$. Thus $g_{1}b_{i-2}\notin E$. The triple ($b_{i-2}, g_{1}, b_{i-1}$) is a bad triple in $\sigma_{1}$.

          We choose $g_{2}$ as $g_{2}$=LMPN$(g_{1}|_{\sigma_{1}} b_{i-1})$ and thus $g_{2}b_{i-2}\in E$. The non-edge $g_{1}a_{i-1}\notin E$ implies that $g_{2}a_{i-1}\in E$. It holds that $g_{2}a_{i-2}\notin E$, since otherwise $\{g_{2}, a_{i-1}, a_{i-2}, b_{i-2}, b_{i-1}\}$ induces a $\overline{P_{2}\cup P_{3}}$, where $P_{2}$ is the path $a_{i-1}-b_{i-2}$ and $P_{3}$ is the path $g_{2}-b_{i-1}-a_{i-2}$, a contradiction. We show (by contradiction) that $a_{i-2}\prec_{1} g_{2}\prec_{1} b_{i-2}$. If $g_{2}\prec_{1} a_{i-2}$, then $g_{2}a_{i-2}\notin E$ implies that $g_{1}a_{i-2}\in E$. Note that $a_{i-1}\prec_{1} g_{1}$ and $a_{i-1}g_{1}\notin E$, thus $g_{1}\prec_{2} a_{i-1}$. Observe that $g_{1}a_{i-2}\in E$ and $g_{1}b_{i-2}\notin E$, contradicting that $a_{i-1}$=LMPN$(a_{i-2}|_{\sigma_{2}} b_{i-2})$. If $g_{2}= a_{i-2}$, then we have $g_{1}a_{i-2}\in E$ and $g_{1}b_{i-2}\notin E$, which leads to the same contradiction. Therefore, $a_{i-2}\prec_{1} g_{2}\prec_{1} b_{i-2}$. Since $g_{2}a_{i-2}\notin E$, $g_{2}b_{i-3}\notin E$. The triple ($b_{i-3}, g_{2}, b_{i-2}$) is a bad triple in $\sigma_{1}$.

          We choose $g_{3}$ as $g_{3}$=LMPN$(g_{2}|_{\sigma_{1}} b_{i-2})$. We deal with $g_{3}$ in the same way as with $g_{2}$, and thus obtain a sequence of vertices $\{g_{j}\}_{2\leq j\leq i-1}$ such that $g_{j}$=LMPN$(g_{j-1}|_{\sigma_{1}} b_{i-j+1})$ and $a_{i-j}\prec_{1} g_{j}\prec_{1} b_{i-j}$, satisfying that $g_{j}b_{i-j-1}\notin E$ (if $j\leq i-2$), $g_{j}a_{i-j}\notin E$, $g_{j}b_{i-j}\in E$ and $g_{j}a_{i-j+1}\in E$, as shown in Fig. \ref{Fig:b(i+1)1}. Especially, $a_{1}\prec_{1} g_{i-1}\prec_{1} b_{1}$, and $g_{i-1}a_{1}\notin E$. Then, ($a_{1}, g_{i-1}, b_{1}$) is a bad triple in $\sigma_{1}$, resulting that there exists a vertex left of $a_{1}$ in $\sigma_{1}$ which is adjacent to $a_{1}$ and $g_{i-1}$ but not to $b_{1}$, contradicting that $S\cap N(a_{1})= S\cap N(b_{1})$.

          %情形2-----------------------------------------------------------------------------------------------------------
          \smallskip
          {\it \bf{Case 2:}} $b_{i}\prec_{1}b_{i+1}\prec_{1}a_{i+1}$.

          \begin{figure}[ht]
          \centering
          \includegraphics[width=6in]{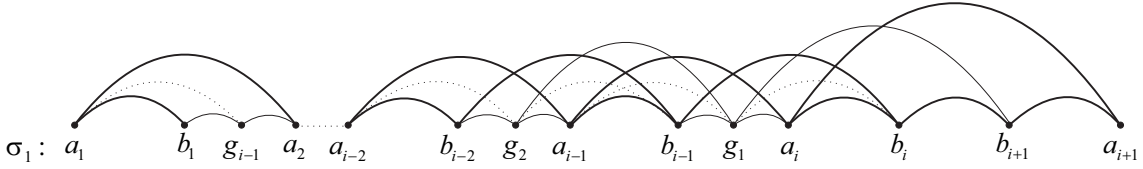}
          \caption{\small Case 2. $b_{i}\prec_{1}b_{i+1}\prec_{1}a_{i+1}$.}\label{Fig:b(i+1)2}
          \end{figure}

          In this case, ($a_{i}, b_{i+1}, a_{i+1}$) is a bad triple in $\sigma_{1}$, so we choose $g_{1}$ as $g_{1}$=LMPN$(b_{i+1}|_{\sigma_{1}} a_{i+1})$ and thus $g_{1}a_{i}\in E$. It follows from the ladder structure that $g_{1}$ can't be any of the vertices $\{a_{j}, b_{j}\}_{1\leq j\leq i}$. It holds that $g_{1}b_{i}\notin E$, since otherwise $\{g_{1}, a_{i}, b_{i}, b_{i+1}, a_{i+1}\}$ induces a $\overline{P_{2}\cup P_{3}}$, where $P_{2}$ is the path $a_{i}-b_{i+1}$ and $P_{3}$ is the path $g_{1}-a_{i+1}-b_{i}$, a contradiction.

          Consider the position of $g_{1}$ in $\sigma_{1}$. We know that $a_{i-1}b_{i+1}, b_{i-1}b_{i+1}\notin E$. If $g_{1}\prec_{1} a_{i-1}$, then $g_{1}a_{i-1}, g_{1}b_{i-1}\in E$. Thus, $\{g_{1}, a_{i}, a_{i-1}, b_{i-1}, b_{i}\}$ induces a $\overline{P_{2}\cup P_{3}}$, where $P_{2}$ is the path $a_{i}-b_{i-1}$ and $P_{3}$ is the path $g_{1}-b_{i}-a_{i-1}$, a contradiction. If $a_{i-1}\prec_{1} g_{1}\prec_{1} b_{i-1}$, then $g_{1}b_{i-1}\in E$ as $b_{i-1}b_{i+1}\notin E$. Thus $g_{1}a_{i-1}\notin E$ because of the same contradiction above, and thus $g_{1}b_{i-2}\notin E$, resulting that ($b_{i-2}, g_{1}, b_{i-1}$) is a bad triple in $\sigma_{1}$, which is the same as in Case 1. Therefore, we assume from now on that $b_{i-1}\prec_{1} g_{1}\prec_{1} a_{i}$. Since $g_{1}b_{i}\notin E$, $g_{1}b_{i-1}\in E$. Similarly, we have $g_{1}a_{i-1}\notin E$. The triple ($a_{i-1}, g_{1}, a_{i}$) is a bad triple in $\sigma_{1}$.

          We choose $g_{2}$ as $g_{2}$ =LMPN$(g_{1}|_{\sigma_{1}} a_{i})$ and thus $g_{2}a_{i-1}\in E$. Since $b_{i-2}a_{i-1}\notin E$, $g_{2}\neq b_{i-2}$. It holds that $g_{2}b_{i-1}\notin E$, since otherwise $\{g_{2}, a_{i-1}, b_{i-1}, g_{1}, a_{i}\}$ induces a $\overline{P_{2}\cup P_{3}}$, where $P_{2}$ is the path $a_{i-1}-g_{1}$ and $P_{3}$ is the path $g_{2}-a_{i}-b_{i-1}$, a contradiction.

          In the following we consider the position of $g_{2}$ in $\sigma_{1}$. If $g_{2}\prec_{1} a_{i-2}$, then $g_{2}a_{i-2}\in E$. Since otherwise, if $g_{2}a_{i-2}\notin E$, then $g_{1}a_{i-2}\in E$. Note that $g_{1}a_{i-1}\notin E$ and $b_{i-2}\prec_{1} a_{i-1}\prec_{1} g_{1}$ imply $g_{1}b_{i-2}\notin E$. Since $a_{i-1}\prec_{1} g_{1}$ and $a_{i-1}g_{1}\notin E$, we have $g_{1}\prec_{2} a_{i-1}$, contradicting to the choice of $a_{i-1}$ as $a_{i-1}$=LMPN$(a_{i-2}|_{\sigma_{2}} b_{i-2})$. Thus we have $g_{2}a_{i-2}\in E$. Since $g_{1}b_{i-2}\notin E$, $g_{2}b_{i-2}\in E$. Then $\{g_{2}, a_{i-1}, a_{i-2}, b_{i-2}, b_{i-1}\}$ induces a $\overline{P_{2}\cup P_{3}}$, where $P_{2}$ is the path $a_{i-1}-b_{i-2}$ and $P_{3}$ is the path $g_{2}-b_{i-1}-a_{i-2}$, a contradiction. If $g_{2}= a_{i-2}$, then immediately we have $g_{1}a_{i-2}\in E$ and $g_{1}b_{i-2}\notin E$, which still contradicts that $a_{i-1}$=LMPN$(a_{i-2}|_{\sigma_{2}} b_{i-2})$. If $a_{i-2}\prec_{1} g_{2}\prec_{1} b_{i-2}$, then $g_{2}a_{i-1}\in E$ (by the triple ($g_{2}, a_{i-1}, g_{1}$)) and $b_{i-2}a_{i-1}\notin E$ imply that $g_{2}b_{i-2}\in E$ (otherwise, ($g_{2}, b_{i-2}, a_{i-1}$) would be a bad triple in $\sigma_{1}$). It holds that $g_{2}a_{i-2}\notin E$, since otherwise $\{g_{2}, b_{i-2}, a_{i-2}, a_{i-1}, b_{i-1}\}$ induces a $\overline{P_{2}\cup P_{3}}$, where $P_{2}$ is the path $b_{i-2}-a_{i-1}$ and $P_{3}$ is the path $g_{2}-b_{i-1}-a_{i-2}$, a contradiction. If $i> 3$, $b_{i-3}a_{i-2}\notin E$, $g_{2}b_{i-3}\notin E$ implies ($b_{i-3}, g_{2}, b_{i-2}$) is a bad triple in $\sigma_{1}$, which is the same as in Case 1. If $i=3$, then ($a_{1}=a_{i-2}, g_{2}, b_{i-2}=b_{1}$) is a bad triple in $\sigma_{1}$, contradicting that $S\cap N(a_{1})=S\cap N(b_{1})$ by using the LBFS $C_{4}$ Property. Therefore we assume from now on that $b_{i-2}\prec_{1} g_{2}\prec_{1} a_{i-1}$.

          Since $g_{2}b_{i-1}\notin E$, it follows that $g_{2}b_{i-2}\in E$. If $g_{2}a_{i-2}\in E$, then $\{g_{2}, a_{i-1}, a_{i-2}, b_{i-2}, b_{i-1}\}$ induces a $\overline{P_{2}\cup P_{3}}$, where $P_{2}$ is the path $a_{i-1}-b_{i-2}$ and $P_{3}$ is the path $g_{2}-b_{i-1}-a_{i-2}$, a contradiction. Thus $g_{2}a_{i-2}\notin E$. Then ($a_{i-2}, g_{2}, a_{i-1}$) is a bad triple in $\sigma_{1}$. Choose $g_{3}$ as $g_{3}$=LMPN $(g_{2}|_{\sigma_{1}} a_{i-1})$. We deal with $g_{3}$ in the same way as with $g_{2}$, and thus obtain a sequence of vertices $\{g_{j}\}_{2\leq j\leq i-1}$, such that $g_{j}$=LMPN$(g_{j-1}|_{\sigma_{1}} a_{i-j+2})$ and $b_{i-j}\prec_{1} g_{j}\prec_{1} a_{i-j+1}$, satisfying $g_{j}a_{i-j+1}\in E, g_{j}b_{i-j+1}\notin E, g_{j}b_{i-j}\in E$ and $g_{j}a_{i-j}\notin E$, as shown in Fig. \ref{Fig:b(i+1)2}. Especially, $b_{1}\prec_{1} g_{i-1}\prec_{1} a_{2}$, $g_{i-1}b_{1}\in E$ and ($a_{1}, g_{i-1}, a_{2}$) is a bad triple in $\sigma_{1}$. Thus there exists a vertex $c\prec_{1} a_{1}$ (thus $c\in S$) such that $ca_{1}, cg_{i-1}\in E$ and $ca_{2}\notin E$. Since $S\cap N(a_{1})=S\cap N(b_{1})$, it follows that $cb_{1}\in E$. Then $\{c, g_{i-1}, b_{1}, a_{1}, a_{2}\}$ induces a $\overline{P_{2}\cup P_{3}}$, where $P_{2}$ is the path $g_{i-1}-a_{1}$ and $P_{3}$ is the path $c-a_{2}-b_{1}$, a contradiction.

          Thus we obtain that $b_{i+1}$ must be placed in $\sigma_{1}$ as $a_{i+1}\prec_{1} b_{i+1}$, as required in condition (3). Therefore, we have completely proved the correctness of Claim 1.

\smallskip
  Since we can always find such a sequence of vertices $a_{1}, a_{2}, ..., a_{t}$, $b_{1}, b_{2}, ..., b_{t}$ for any integer $t\geq 2$, we get a contradiction to $G$ being finite. Thus $\sigma_{1}=\sigma_{3}$, as required.
\end{proof}

Combining Lemma \ref{Repeated LBFS+} with Theorem \ref{result}, we immediately obtain our main result as following.

\begin{thm}\label{LexCycle-result}
  Let $G$ be a $\overline{P_{2}\cup P_{3}}$-free cocomparability graph. Then $LexCycle(G)=2$.
\end{thm}

Note that $\overline{P_{2}\cup P_{3}}$-free cocomparability graphs strictly contain both $C_{4}$-free cocomparability graphs (i.e., interval graphs, which have been proved in \cite{LexCycle}) and diamond-free cocomparability graphs, where a diamond consists of a complete graph $K_{4}$ minus one edge, we thus immediately obtain the following corollaries.

\begin{cor}{\rm \cite{LexCycle}}
  Let $G$ be an interval graph. Then $LexCycle(G)=2$.
\end{cor}

\begin{cor}\label{diamond}
  Let $G$ be a diamond-free cocomparability graph. Then $LexCycle(G)=2$.
\end{cor}

Additionally, $\overline{P_{2}\cup P_{3}}$-free cocomparability graphs strictly contain triangle-free cocomparability graphs, we thus immediately obtain that this result also holds for cocomparability graphs with girth at least 4.

\begin{cor}\label{girth}
  Let $G$ be a cocomparability graph with girth $g(G)\geq 4$. Then $LexCycle(G)=2$.
\end{cor}

\section{Concluding remarks}
\label{sec:pdf}
In this paper we focus on the parameter called LexCycle($G$), recently introduced by Charbit et al. \cite{LexCycle}, and show that LexCycle($G$)=2 if $G$ is a $\overline{P_{2}\cup P_{3}}$-free cocomparability graph. As corollaries, it's applicable for diamond-free cocomparability graphs, cocomparability graphs with girth at least 4, as well as interval graphs. In the proof of Theorem \ref{result}, we have assumed that $b_{i+1}$=LMPN$(b_{i}|_{\sigma_{0}} a_{i})$. In fact, using this requirement, we can get the strict ordering of $a_{1}, a_{2}, ..., a_{i}, b_{1}, b_{2}, ..., b_{i}$ in $\sigma_{3}$ as $b_{1}\prec_{3} a_{1}\prec_{3} b_{2}\prec_{3} a_{2}\prec_{3} ... \prec_{3} b_{i}\prec_{3} a_{i}$.\\[3ex]
{\bf Acknowledgements} The authors would like to thank anonymous reviewers for their helpful comments and suggestions which lead to a considerably improved presentation.

\nocite{*}
\bibliographystyle{abbrvnat}
% use the following instead if you encounter problems
%\bibliographystyle{alpha}
\bibliography{sample-dmtcs}
\label{sec:biblio}

\end{document}